\documentclass[fleqn,preprint,3p,a4paper]{elsarticle}

\usepackage[T1]{fontenc}
\usepackage[utf8]{inputenc}
\usepackage{amsmath,amssymb,amsthm,mathtools,mathrsfs}
\usepackage{enumerate}
\usepackage{microtype}
\usepackage[hidelinks]{hyperref}

\theoremstyle{plain}
\newtheorem{theorem}{Theorem}[section]
\newtheorem{lemma}[theorem]{Lemma}
\newtheorem{proposition}[theorem]{Proposition}
\newtheorem{corollary}[theorem]{Corollary}
\theoremstyle{remark}
\newtheorem{remark}[theorem]{Remark}

\newcommand{\KIrr}{\operatorname{K\text{-}Irr}}
\newcommand{\Irrc}{\operatorname{Irr}_{c}}
\newcommand{\Up}{\mathord{\uparrow}}
\newcommand{\calO}{\mathcal{O}}

\begin{document}

\begin{frontmatter}
\title{Products of co-sober spaces need not be co-sober\tnoteref{t1}}
\tnotetext[t1]{This research was supported by the National Natural Science Foundation of China (Nos. 12471070, 12071199, 12561083).}

\author[nit]{Xiaoquan Xu\corref{cor}}
\ead{xiqxu2002@163.com}
\cortext[cor]{Corresponding author}

\author[glut]{Wei Ji}
\ead{javeey@163.com}

\address[nit]{School of Computer Information Engineering, Nanchang Institute of Technology, Nanchang 330044, China}
\address[glut]{School of Mathematics and Statistics, Guilin University of Technology, Guilin 541004, China}

\begin{abstract}
We answer negatively the question whether the product of two co-sober spaces must be co-sober. For a nonempty closed subset $A$ of a $T_1$-space $(X,\tau)$, we consider the topology consisting of the co-finite open sets together with the $\tau$-open sets disjoint from $A$. We determine its $k$-irreducible compact subsets and derive a criterion that produces two compact co-sober spaces whose product contains a non-singleton $k$-irreducible compact saturated subset of the diagonal. Applying the criterion to a specific Kov\'a\v{r}--van Douwen space gives two countable compact Fr\'echet--Urysohn $T_1$ anti-Hausdorff co-sober spaces whose product is not co-sober. A further construction yields a countable compact Fr\'echet--Urysohn $T_1$ anti-Hausdorff co-sober space whose square is not co-sober. The failure of product closure also gives an alternative proof of the known non-reflectivity of co-sober spaces in the category of $T_0$-spaces.
\end{abstract}
\end{frontmatter}

\section{Introduction}

Co-sobriety was introduced by Escard\'o, Lawson and Simpson in connection with a dual form of the Hofmann--Mislove theorem \cite{ELS}. It asks, in a precise sense, whether the indecomposable compact saturated subsets of a $T_0$-space are generated by points: a $T_0$-space is co-sober if every $k$-irreducible compact saturated set is the saturation of a point. Although this definition is naturally paired with sobriety, the two notions are not formal duals; in particular, sober spaces need not be co-sober \cite{ShenWuXiZhao2020, WenXu}. Product stability is a basic structural question for co-sobriety. Xu and Zhao asked whether an arbitrary product of co-sober spaces must be co-sober \cite[Question~9.8]{XuZhao}. 

We prove that the binary product property fails. More precisely, there are two countable compact Fr\'echet--Urysohn $T_1$ anti-Hausdorff co-sober spaces whose product is not co-sober. The construction begins with a co-finite modification of a $T_1$ topology. Given a nonempty closed set $A\subseteq X$, we retain all original open sets disjoint from $A$ and all co-finite open sets. The resulting topology is compact and $T_1$, and its $k$-irreducible compact subsets can be described directly. Two disjoint closed sets meeting every non-singleton $k$-irreducible compact subset then give two co-sober modifications whose join is the original topology. The original obstruction reappears on the diagonal of their product.

The specialized input is the particular $\omega$-MAD realization of van Douwen's compact Fr\'echet--Urysohn anti-Hausdorff US-space considered by Kov\'a\v{r} \cite{vanDouwen,Kovar}. Kov\'a\v{r}'s description of its proper closed sets implies that the whole space is its only non-singleton $k$-irreducible compact subset. This conclusion is stronger than the usual list of properties of van Douwen's example and is the only specialized external input in the construction.

We then combine the two factors by adjoining a countable co-finite core. This produces a single countable compact Fr\'echet--Urysohn $T_1$ anti-Hausdorff co-sober space whose square is not co-sober; thus the failure of product stability already occurs for the square of one space. Since a full reflective subcategory of the category of $T_0$-spaces that is closed under homeomorphisms must be closed under products, this gives an alternative proof of the non-reflectivity of co-sober spaces, established by Shen, Xi and Zhao \cite[Corollary~4.3]{ShenXiZhao2024}. The categorical conclusion is not new; the point here is that it follows directly from failure of closure under squares. For comparison, we conclude with a positive result showing that the product of two co-sober spaces is co-sober whenever one factor is a KC-space.

Section~\ref{sec:prelim} recalls the terminology. Section~\ref{sec:modification} studies the co-finite modification. Section~\ref{sec:criterion} gives the product criterion. Section~\ref{sec:kv} verifies the Kov\'a\v{r}--van Douwen input and constructs the binary counterexample. Section~\ref{sec:square} constructs the square counterexample and derives the non-reflectivity consequence. Section~\ref{sec:kc} treats products with a KC factor.

\section{Preliminaries}\label{sec:prelim}

All compactness statements are made without a separation assumption. Let $X$ be a $T_0$-space. Its specialization order is denoted by $\leq_X$, and the saturation of a point $x$ is written $\Up x$. A subset is saturated if it is an upper set for the specialization order.

Let $Q(X)$ be the family of all nonempty compact saturated subsets of $X$. A member $K\in Q(X)$ is \emph{$k$-irreducible} if
\[
 K=K_1\cup K_2,\quad K_1,K_2\in Q(X),
\]
implies $K=K_1$ or $K=K_2$. We write $\KIrr(X)$ for the family of all such sets, and $\Irrc(X)$ for the family of all nonempty irreducible closed subsets of $X$. When the topology needs to be displayed explicitly, we write, for example, $Q(X,\sigma)$, $\KIrr(X,\sigma)$ and $\Irrc(X,\sigma)$. The space $X$ is \emph{co-sober} if every $K\in\KIrr(X)$ has the form $\Up x$ for a point $x\in X$ \cite{ELS}. The representing point is unique because $X$ is $T_0$.

In a $T_1$-space every subset is saturated and $\Up x=\{x\}$. Hence a $T_1$-space is co-sober precisely when its $k$-irreducible compact subsets are singletons.

A topological space $X$ is called a \emph{Fr\'echet--Urysohn space}  if for every $A\subseteq X$ and $x\in \overline{A}$, there exists a sequence $(x_n)_{n\in\mathbb N}$ in $A$ converging to $x$. We shall use the standard fact that every first-countable space is Fr\'echet (see, e.g., \cite[Theorem~1.6.14]{Engelking}).

As in \cite{vanDouwen, KunziZypen}, a topological space $X$ is called a \emph{KC}-\emph{space} if each compact set of $X$ is closed. The space $X$ is said to be a \emph{US}-\emph{space} provided that each convergent sequence has a unique limit. It is known that each Hausdorff space (=$T_2$-space) is a KC-space, each KC-space is a US-space, each US-space is a $T_1$-space and each first-countable \emph{US}-space is a $T_2$-space (cf. \cite{Wilansky-1967}).

\begin{lemma}\label{lem-Frechet-US-KC} \emph{(\cite[Lemma 1]{KunziZypen})}  Every Fr\'echet--Urysohn US-space is a KC-space.
\end{lemma}

A topological space $X$ is called \emph{anti}-\emph{Hausdorff} (see, e.g., \cite{Hoffmann-1979-1}) if each pair of nonempty open sets of $X$ intersects.

\begin{lemma}\label{lem:diagonal}
For topologies $\tau_1$ and $\tau_2$ on the same set $X$, let $\tau_1\vee\tau_2$ denote their join in the lattice of topologies on $X$. Then the diagonal $\Delta(X)=\{(x, x) : x\in X\}$ of $(X,\tau_1)\times(X,\tau_2)$, with its subspace topology, is homeomorphic to $(X,\tau_1\vee\tau_2)$.
\end{lemma}
\begin{proof} Let $\tau$ be the product topology of $\tau_1$ and $\tau_2$. Then $\{U_1\times U_2 : U_1\in \tau_1, U_2\in\tau_2\}$ is a base of $\tau$ and $\{U_1\cap U_2 : U_1\in \tau_1, U_2\in\tau_2\}$ is a base of $\tau_1\vee\tau_2$.

Define
\[
 \varphi:(X,\tau_1\vee\tau_2)\longrightarrow
 (\Delta(X),\tau|_{\Delta(X)}),\quad x\longmapsto(x,x).
\]
Clearly, $\varphi$ is a bijection. For any $W_1\in \tau_1, W_2\in\tau_2$, we have $\varphi^{-1}((W_1\times W_2)\cap \Delta(X))=W_1\cap W_2$ or, equivalently, $\varphi(W_1\cap W_2)=(W_1\times W_2)\cap \Delta(X)$. Hence $\varphi$ is a homeomorphism.
\end{proof}

\section{A co-finite modification}\label{sec:modification}

Let $(X,\tau)$ be a $T_1$-space, and let $A\subseteq X$ be nonempty and closed. Write $\tau_{\mathrm{cof}}$ for the co-finite topology on $X$ and define
\[
 \tau_A=\tau_{\mathrm{cof}}\cup\{U\in\tau:U\cap A=\varnothing\}.
 \tag{3.1}
\]
Thus a $\tau_A$-open set meeting $A$ is co-finite, while every $\tau$-open set disjoint from $A$ is retained.

\begin{lemma}\label{lem:basic}
The family $\tau_A$ is a compact $T_1$ topology contained in $\tau$. It has the following additional properties.
\begin{enumerate}[\rm (1)]
\item Every subset of $X$ meeting $A$ is compact in $(X,\tau_A)$.
\item The restrictions of $\tau_A$ and $\tau$ to $X\setminus A$ coincide.
\item A subset $K\subseteq X$ is compact in $(X,\tau_A)$ if and only if either $K\cap A\neq\varnothing$, or $K\cap A=\varnothing$ and $K$ is compact in $(X,\tau)$.
\end{enumerate}
\end{lemma}

\begin{proof}
Since $(X,\tau)$ is $T_1$, every co-finite subset of $X$ is $\tau$-open, so $\tau_A\subseteq\tau$. If a union of members of $\tau_A$ meets $A$, one of its members meets $A$ and is co-finite; hence the union is co-finite. If a finite intersection meets $A$, every factor containing the common point is co-finite, so the intersection is co-finite. If the intersection misses $A$, it is a $\tau$-open set disjoint from $A$. Thus $\tau_A$ is a topology. It contains the co-finite topology and is therefore $T_1$.

Let $K\cap A\neq\varnothing$, and let $\mathcal U$ be a $\tau_A$-open cover of $K$. Choose $a\in K\cap A$ and $U\in\mathcal U$ with $a\in U$. The set $U$ is co-finite, so finitely many further members of $\mathcal U$ cover the remaining points of $K$. This proves (1), and the case $K=X$ proves compactness of $(X,\tau_A)$.

For (2), if $U\in\tau$, then $U\setminus A$ is $\tau$-open because $A$ is closed, and it belongs to $\tau_A$. Since
\[
 U\cap(X\setminus A)=(U\setminus A)\cap(X\setminus A),
\]
the restrictions coincide. Assertion (3) follows from (1) and (2).
\end{proof}

\begin{theorem}\label{thm:classification}
Let $(X,\tau)$ be $T_1$, and let $A\subseteq X$ be nonempty and closed. Then
\[
 \KIrr(X,\tau_A)
 =\{\{a\}:a\in A\}
 \cup\{K\in\KIrr(X,\tau):K\cap A=\varnothing\}.
 \tag{3.2}
\]
Consequently, $(X,\tau_A)$ is co-sober if and only if $A$ meets every non-singleton member of $\KIrr(X,\tau)$.
\end{theorem}

\begin{proof}
Let $K\in\KIrr(X,\tau_A)$. Suppose first that $K\cap A\neq\varnothing$ and $|K|>1$. Choose $a\in K\cap A$ and $x\in K\setminus\{a\}$. By Lemma~\ref{lem:basic}, $K\setminus\{x\}$ is compact because it still meets $A$, while $\{x\}$ is compact. Both sets are saturated because $(X,\tau_A)$ is $T_1$. Thus
\[
 K=\{x\}\cup(K\setminus\{x\})
\]
is a union of two proper members of $Q(X,\tau_A)$, a contradiction. Hence every member of $\KIrr(X,\tau_A)$ meeting $A$ is a singleton contained in $A$.

Now suppose that $K\cap A=\varnothing$. By Lemma~\ref{lem:basic}, the subspace topology on $K$, and on every subset of $K$, is the same for $\tau_A$ and $\tau$. Since both spaces are $T_1$, saturation is automatic. Therefore $K$ is compact and $k$-irreducible for $\tau_A$ if and only if it is compact and $k$-irreducible for $\tau$. This proves (3.2), and the final assertion follows from the $T_1$ characterization of co-sobriety.
\end{proof}

\begin{lemma}\label{lem:preservation}
If $(X,\tau)$ is Fr\'echet--Urysohn, then $(X,\tau_A)$ is Fr\'echet--Urysohn. If $(X,\tau)$ is anti-Hausdorff, then $(X,\tau_A)$ is anti-Hausdorff.
\end{lemma}

\begin{proof}
The anti-Hausdorff assertion follows from $\tau_A\subseteq\tau$.

Let $x\in\overline C^{\tau_A}$. If $x\in A$, then either $x\in C$, in which case the constant sequence suffices, or $C$ is infinite; otherwise the co-finite neighborhood $X\setminus C$ of $x$ would miss $C$. In the latter case, choose pairwise distinct points $c_n\in C$. Every $\tau_A$-neighborhood of $x$ is co-finite, so $c_n\to x$.

If $x\notin A$, then for every $\tau$-neighborhood $U$ of $x$, the set $U\setminus A$ is a $\tau_A$-neighborhood of $x$ and therefore meets $C$. Hence
\[
 x\in\overline{C\setminus A}^{\tau}.
\]
The Fr\'echet--Urysohn property of $(X,\tau)$ gives a sequence in $C\setminus A$ converging to $x$ in $\tau$. It also converges in the coarser topology $\tau_A$.
\end{proof}

\section{A product criterion}\label{sec:criterion}

\begin{theorem}\label{thm:criterion}
Let $(X,\tau)$ be a $T_1$-space with a non-singleton $k$-irreducible compact subset. Suppose that there are disjoint nonempty closed sets $A,B\subseteq X$ such that every non-singleton member of $\KIrr(X,\tau)$ meets both $A$ and $B$. Then $(X,\tau_A)$ and $(X,\tau_B)$ are compact $T_1$ co-sober spaces, whereas their product is not co-sober.
\end{theorem}

\begin{proof}
Lemma~\ref{lem:basic} and Theorem~\ref{thm:classification} show that both modified spaces are compact, $T_1$ and co-sober.

We claim that
\[
 \tau_A\vee\tau_B=\tau.
 \tag{4.1}
\]
The inclusion from left to right follows from $\tau_A,\tau_B\subseteq\tau$. Conversely, for every $U\in\tau$,
\[
 U=(U\setminus A)\cup(U\setminus B),
\]
because $A\cap B=\varnothing$. Since $A$ and $B$ are closed, $U\setminus A\in\tau_A$ and $U\setminus B\in\tau_B$. This proves (4.1).

Choose a non-singleton $C\in\KIrr(X,\tau)$. By Lemma \ref{lem:diagonal},
\[
 \Delta(C)=\{(x,x):x\in C\}
\]
is homeomorphic to $C$, and hence compact. Suppose that $\Delta(C)=K_1\cup K_2$, where $K_1$ and $K_2$ are compact saturated subsets of the product. Since $K_i\subseteq \Delta(C)$, their inverse images under the diagonal homeomorphism are compact subsets of $(X,\tau)$. They are saturated because $(X,\tau)$ is $T_1$. The k-irreducibility of $C$ therefore gives $K_1=\Delta(C)$ or $K_2=\Delta(C)$. Thus $\Delta(C)$ is $k$-irreducible in the product. The product is $T_1$, so $\Delta(C)$ is saturated, and its having more than one point prevents it from being the saturation of a point. Hence the product is not co-sober.
\end{proof}

\section{The Kov\'a\v{r}--van Douwen example}\label{sec:kv}

The next lemma extracts the property of the example that is needed here. It concerns the particular $\omega$-MAD realization described by Kov\'a\v{r}, not every space having only the standard properties of van Douwen's example.

\begin{lemma}\label{lem:kv}
There exists a countably infinite compact Fr\'echet--Urysohn anti-Hausdorff KC-space $(M,\tau)$ such that
\[
 \KIrr(M,\tau)=\{\{x\}:x\in M\}\cup\{M\}.
 \tag{5.1}
\]
\end{lemma}

\begin{proof}
Kov\'a\v{r} \cite[p.~2]{Kovar} considers a structure $(M,\mathcal T)$ in which $\mathcal T$ is closed under finite intersections and nonempty unions but need not contain $\varnothing$. He associates with it the ordinary topology
\[
 \tau=\mathcal T\cup\{\varnothing\}.
\]
Let
\[
 \mathcal C=\{H\subseteq M:M\setminus H\in\mathcal T\}.
\]
The closed sets of $(M,\tau)$ are precisely the members of $\mathcal C$ together with $M$.

In \cite[Example~2.2]{Kovar}, Kov\'a\v{r} constructs $\mathcal C$ from a family $\mathcal A$ and a map $l$ satisfying $l(A)\notin A$. In the proof, for every infinite $H\in\mathcal C$ an infinite subset $D\subseteq H$ is obtained such that $H\setminus\{x\}\in\mathcal C$ for every $x\in D$. Choosing distinct $p,q\in D$, both $H\setminus\{p\}$ and $H\setminus\{q\}$ are proper closed subsets and
\[
 H=(H\setminus\{p\})\cup(H\setminus\{q\}).
\]
Hence every infinite member of $\mathcal C$ is reducible. In the paragraph following the example, he takes $M=\mathbb N$ and an $\omega$-MAD family and records that the associated topology is anti-Hausdorff and Fr\'echet--Urysohn, that convergent sequences have unique limits, and that the topology is maximal compact \cite[Example~2.2 and the following paragraph, p.~6]{Kovar}; see also \cite{vanDouwen}.

Thus $(M,\tau)$ is countably infinite, compact, Fr\'echet--Urysohn, anti-Hausdorff and US, and therefore $T_1$. Every infinite proper closed set belongs to $\mathcal C$ and is reducible. Every finite closed set with at least two points is reducible because the space is $T_1$. The whole space $M$ is irreducible because it is anti-Hausdorff. Consequently,
\[
 \Irrc(M,\tau)=\{\{x\}:x\in M\}\cup\{M\}.
 \tag{5.2}
\]

By Lemma \ref{lem-Frechet-US-KC}, $(M,\tau)$ is a KC-space. Since it is compact, its compact subsets are exactly its closed subsets. For a compact subset of $M$, k-irreducibility is therefore equivalent to irreducibility as a closed set. Saturation is automatic because $M$ is $T_1$; compact pieces are closed by the KC property, and closed pieces are compact because $M$ is compact. Combining this observation with (5.2) gives (5.1).
\end{proof}

\begin{remark}\label{rem:scope}
Equation~(5.1) does not follow merely from compactness, the Fr\'echet--Urysohn and US properties, and anti-Hausdorffness. It uses Kov\'a\v{r}'s description of the proper closed sets in the chosen $\omega$-MAD realization.
\end{remark}

\begin{theorem}\label{thm:binary}
There exist two countable compact Fr\'echet--Urysohn $T_1$ anti-Hausdorff co-sober spaces whose product is not co-sober.
\end{theorem}

\begin{proof}
Take $(M,\tau)$ from Lemma~\ref{lem:kv}, choose distinct points $a,b\in M$, and put $A=\{a\}$ and $B=\{b\}$. The only non-singleton member of $\KIrr(M,\tau)$ is $M$, which meets both $A$ and $B$. Theorem~\ref{thm:criterion} gives compact $T_1$ co-sober spaces
\[
 P=(M,\tau_A),~ Q=(M,\tau_B)
\]
such that $P\times Q$ is not co-sober. They are countable, and Lemma~\ref{lem:preservation} shows that they remain Fr\'echet--Urysohn and anti-Hausdorff.
\end{proof}

\section{A co-sober space whose square is not co-sober}\label{sec:square}

Let $P$ and $Q$ be disjoint nonempty spaces, and let $R$ be an infinite set disjoint from both. On
\[
 Z=P\sqcup Q\sqcup R
\]
define
\[
 \lambda=\{\varnothing\}\cup
 \left\{U\cup W\cup(R\setminus F):
 U\in\calO(P),\ W\in\calO(Q),\ F\in [R]^{<\omega}\right\},
 \tag{6.1}
\]
where $[R]^{<\omega}$ denotes the family of all finite subsets of $R$.
\begin{lemma}\label{lem:core}
The family $\lambda$ is a topology. If $P$ and $Q$ are $T_1$, then $(Z,\lambda)$ is $T_1$ and anti-Hausdorff, and $P$ and $Q$ are closed subspaces with their original topologies. Moreover:
\begin{enumerate}[\rm (1)]
\item A subset $K\subseteq Z$ is compact if and only if $K\cap P$ is compact in $P$ and $K\cap Q$ is compact in $Q$.
\item If $P$ and $Q$ are compact Fr\'echet--Urysohn spaces, then $Z$ is compact and Fr\'echet--Urysohn.
\item If $P$ and $Q$ are $T_1$ co-sober spaces, then $Z$ is co-sober.
\end{enumerate}
\end{lemma}

\begin{proof}
Finite intersections and nonempty unions of co-finite subsets of the infinite set $R$ are co-finite. Together with the topology axioms on $P$ and $Q$, this proves that $\lambda$ is a topology.

Assume that $P$ and $Q$ are $T_1$. The complement of $p\in P$ is
\[
 (P\setminus\{p\})\cup Q\cup R,
\]
and is open; the same holds for points of $Q$. For $r\in R$, the complement $P\cup Q\cup(R\setminus\{r\})$ is open. Thus $Z$ is $T_1$. Every nonempty open set contains a co-finite subset of $R$, so any two nonempty open sets intersect. For $U\in\calO(P)$, the set $U\cup R$ is open in $Z$ and has trace $U$ on $P$, and similarly for $Q$. Finally, $Z\setminus P=Q\cup R$ and $Z\setminus Q=P\cup R$ are open, so both subspaces are closed.

To prove (1), suppose first that $K$ is compact. If $\mathcal U$ is an open cover of $K\cap P$ in $P$, then
\[
 \{U\cup R:U\in\mathcal U\}\cup\{Q\cup R\}
\]
is an open cover of $K$ in $Z$. A finite subcover shows that $K\cap P$ is compact. The argument for $K\cap Q$ is symmetric.

Conversely, assume that the two traces are compact and let $\mathcal V$ be an open cover of $K$. Finitely many members cover $K\cap P$, and finitely many cover $K\cap Q$. If some member has been chosen, their union contains all but finitely many points of $R$, and finitely many further members cover the remaining points of $K\cap R$. If both traces are empty and $K$ is nonempty, choose one member of $\mathcal V$ meeting $K\subseteq R$; it covers all but finitely many points of $K$, and finitely many additional members complete the cover. The empty set is compact. This proves (1), and compactness in (2) follows by taking $K=Z$.

For the Fr\'echet--Urysohn assertion, let $z\in\overline C^Z$. If $C\cap R$ is infinite, a sequence of distinct points of $C\cap R$ converges to every point of $Z$, because every nonempty neighborhood contains all but finitely many points of $R$. Suppose that $C\cap R$ is finite. If $z\in P$ and $z\notin\overline{C\cap P}^{P}$, choose an open neighborhood $U$ of $z$ in $P$ disjoint from $C\cap P$. Then
\[
 U\cup\bigl(R\setminus(C\cap R)\bigr),
\]
with empty $Q$-part, is an open neighborhood of $z$ in $Z$ disjoint from $C$, a contradiction. Hence $z\in\overline{C\cap P}^{P}$, and the Fr\'echet--Urysohn property of $P$ supplies a sequence in $C\cap P$ converging to $z$. The case $z\in Q$ is symmetric. If $z\in R$ and $z\notin C$, then $R\setminus(C\cap R)$, with empty $P$- and $Q$-parts, is an open neighborhood of $z$ disjoint from $C$. Thus $z\in C$, and the constant sequence suffices. This proves (2).

For (3), let $K$ be a $k$-irreducible compact subset of $Z$. Saturation is automatic because $Z$ is $T_1$. If $K$ meets both $P$ and $Z\setminus P$, then
\[
 K=(K\cap P)\cup(K\setminus P)
\]
is a union of two nonempty proper compact subsets by (1), a contradiction. Repeating the argument with $Q$ shows that $K$ is contained in one of $P$, $Q$ or $R$. If $K\subseteq P$, then (1) and the definition of k-irreducibility show that $K$ is $k$-irreducible in $P$, so it is a singleton. The same holds in $Q$. If $K\subseteq R$ and $|K|>1$, then $K=\{r\}\cup(K\setminus\{r\})$ is a union of two proper compact subsets by (1). Hence every $k$-irreducible compact subset of $Z$ is a singleton, and $Z$ is co-sober.
\end{proof}

\begin{theorem}\label{thm:square}
There exists a countable compact Fr\'echet--Urysohn $T_1$ anti-Hausdorff co-sober space whose square is not co-sober.
\end{theorem}

\begin{proof}
Let $(M,\tau)$ and the distinct points $a,b\in M$ be as in the proof of Theorem~\ref{thm:binary}. Put $A=\{a\}$ and $B=\{b\}$. To obtain disjoint copies of the two factors, set
\[
 P=M\times\{0\},\quad Q=M\times\{1\},
\]
and transport the topologies $\tau_A$ and $\tau_B$ to $P$ and $Q$, respectively. Take $R$ countably infinite and disjoint from $P\sqcup Q$, and define $Z$ by (6.1). Lemma~\ref{lem:core} shows that $Z$ has all the asserted properties.

Define
\[
 \Gamma=\{((x,0),(x,1)):x\in M\}\subseteq P\times Q\subseteq Z^2.
\]
By (4.1) and Lemma~\ref{lem:diagonal}, the map
\[
 h:(M,\tau)\longrightarrow\Gamma,\quad
 h(x)=((x,0),(x,1)),
\]
is a homeomorphism. Since $M\in\KIrr(M,\tau)$ by Lemma~\ref{lem:kv}, the set $\Gamma$ is compact. If
\[
 \Gamma=K_1\cup K_2
\]
with $K_1,K_2$ compact saturated subsets of $Z^2$, then $K_i\subseteq\Gamma$. Thus $h^{-1}(K_i)$ is compact in $(M,\tau)$; it is also saturated because $(M,\tau)$ is $T_1$. The $k$-irreducibility of $M$ gives $h^{-1}(K_i)=M$ for some $i$, and hence $K_i=\Gamma$. Therefore $\Gamma$ is $k$-irreducible in $Z^2$. Since $Z^2$ is $T_1$ and $\Gamma$ has more than one point, it is not the saturation of a point. Consequently, $Z^2$ is not co-sober.
\end{proof}

\begin{remark}\label{rem:powers}
The same witness placed in any two coordinates shows that $Z^I$ is not co-sober whenever $|I|\geq2$; the other coordinates may be fixed at one point of $Z$. This is an immediate extension of Theorem~\ref{thm:square}, not an additional construction. The added core does not preserve the US property: every injective sequence in $R$ converges to every point of $Z$. No assertion is made that $Z^2$ is Fr\'echet--Urysohn.
\end{remark}

The square counterexample also has a categorical consequence. Let $\mathbf{Top}_0$ denote the category of $T_0$-spaces and continuous maps, and let $\mathbf{CoSob}$ be its full subcategory of co-sober spaces. A $\mathbf{CoSob}$-reflection of a space $X$ is a continuous map $\eta:X\to C$, with $C$ co-sober, through which every continuous map from $X$ to a co-sober space factors uniquely. The subcategory is reflective if every $T_0$-space has such a reflection.

\begin{corollary}[Shen--Xi--Zhao {\cite[Corollary~4.3]{ShenXiZhao2024}}]\label{cor:nonreflective}
The full subcategory $\mathbf{CoSob}$ is not reflective in $\mathbf{Top}_0$.
\end{corollary}

\begin{proof}
Let $Z$ be the co-sober space in Theorem~\ref{thm:square}. We show directly that $Z^2$ has no $\mathbf{CoSob}$-reflection. Suppose that $\eta:Z^2\to C$ were such a reflection. Since $Z$ is co-sober, the coordinate projections $\pi_i:Z^2\to Z$ factor through $\eta$: there are continuous maps $p_i:C\to Z$ satisfying $p_i\circ\eta=\pi_i$ for $i=1,2$. The continuous map
\[
 r:C\longrightarrow Z^2,\qquad r(c)=(p_1(c),p_2(c)),
\]
satisfies $r\circ\eta=\operatorname{id}_{Z^2}$. Therefore
\[
 (\eta\circ r)\circ\eta=\eta=\operatorname{id}_C\circ\eta.
\]
By the uniqueness in the reflection property, applied to the map $\eta:Z^2\to C$, we have $\eta\circ r=\operatorname{id}_C$. Thus $\eta$ is a homeomorphism, so $Z^2$ is co-sober because $C$ is co-sober. This contradicts Theorem~\ref{thm:square}. Hence $Z^2$ has no $\mathbf{CoSob}$-reflection, and $\mathbf{CoSob}$ is not reflective.
\end{proof}

Shen, Xi and Zhao obtained this conclusion from their general necessary condition that a full reflective subcategory of $\mathbf{Top}_0$, closed under homeomorphisms and containing a non-$T_1$ space, must contain every sober space \cite[Theorem~3.6(3)]{ShenXiZhao2024}. The proof above instead uses the explicit failure of closure under binary products, already witnessed by the square in Theorem~\ref{thm:square}, without invoking that general characterization.

\section{Products with a KC factor}\label{sec:kc}

For comparison, we record a positive observation.

\begin{proposition}\label{prop:kc}
Let $X$ and $Y$ be co-sober spaces. If $X$ is a KC-space, then $X\times Y$ is co-sober.
\end{proposition}

\begin{proof}
A KC-space is $T_1$, since every singleton is compact and hence closed. Let $K\in\KIrr(X\times Y)$ and put $A=\pi_X(K)$. The set $A$ is compact. It is also saturated: if $x\in A$ and $x\leq_X x'$, choose $y$ with $(x,y)\in K$. The specialization order of a product is coordinatewise, so $(x',y)\in K$ and $x'\in A$.

Suppose that
\[
 A=A_1\cup A_2,~ A_1,A_2\in Q(X).
\]
Each $A_i$ is closed because $X$ is KC. Hence
\[
 K_i=K\cap(A_i\times Y)
\]
is a closed subset of the compact set $K$, and is therefore compact. It is nonempty because $A_i\subseteq\pi_X(K)$, and it is saturated because both $K$ and $A_i\times Y$ are saturated. Moreover, $K=K_1\cup K_2$. The k-irreducibility of $K$ gives $K=K_i$ for some $i$, and projection gives $A=A_i$. Thus $A\in\KIrr(X)$.

By co-sobriety of $X$, $A=\Up x$ for some $x\in X$. Since $X$ is $T_1$, $A=\{x\}$. Hence
\[
 K=\{x\}\times L,
~ L=\pi_Y(K).
\]
The set $L$ is compact. It is also saturated: if $y\in L$ and $y\leq_Y y'$, choose $x$ with $(x,y)\in K$. Since the specialization order on $X\times Y$ is coordinatewise and $K$ is saturated, $(x,y')\in K$, so $y'\in L$. The set $L$ is $k$-irreducible, because any decomposition of $L$ into compact saturated subsets would give a corresponding decomposition of $K$ after taking the product with $\{x\}$. Since $Y$ is co-sober, $L=\Up y$ for some $y\in Y$. Consequently,
\[
 K=\{x\}\times\Up y=\Up(x,y).
\]
Thus $X\times Y$ is co-sober.
\end{proof}

By Lemma~\ref{lem-Frechet-US-KC}, Proposition~\ref{prop:kc}, and the standard implication ``first-countable $\Rightarrow$ Fr\'echet--Urysohn'', we obtain the following two corollaries.

\begin{corollary}\label{cor:frechet}
Let $X$ and $Y$ be co-sober spaces. If $X$ is a Fr\'echet--Urysohn US-space, then $X\times Y$ is co-sober.
\end{corollary}

\begin{corollary}\label{cor:c1}
Let $X$ and $Y$ be co-sober spaces. If $X$ is a first-countable US-space, then $X\times Y$ is co-sober.
\end{corollary}

\section{Conclusion}

We have shown that co-sobriety is not preserved by binary products. The two factors can be chosen countable, compact, Fr\'echet--Urysohn, $T_1$ and anti-Hausdorff. A related construction gives a single space with the same properties whose square is not co-sober. The general mechanism is the classification in Theorem~\ref{thm:classification} together with the diagonal argument in Theorem~\ref{thm:criterion}. The specialized external input is confined to Lemma~\ref{lem:kv}. The square counterexample also gives a direct product-based proof of the known non-reflectivity theorem of Shen, Xi and Zhao; in fact, the space $Z^2$ in Theorem~\ref{thm:square} admits no reflection into co-sober spaces.



\section*{Competing interests}
The authors have no relevant financial or non-financial interests to disclose.

\section*{Data availability}
No datasets were generated or analyzed during the current study.


\end{document}